\documentclass[a4paper,11pt]{article}

\usepackage{amsmath,amsthm}
\usepackage{amssymb}
\usepackage{breqn}
\usepackage{enumerate}
\usepackage{graphicx}
\usepackage{bm}
\usepackage{bbm}
\usepackage[affil-it]{authblk}
\usepackage{tabu}
\usepackage{bold-extra}
\usepackage[hang,flushmargin]{footmisc}
\usepackage[hypertex]{hyperref}
\usepackage{mathtools}
\usepackage{relsize}
\usepackage{scalerel}
\usepackage{xcolor}
\usepackage{titlesec}
\usepackage{apptools}
\usepackage{appendix}

\newtheorem{theorem}{Theorem}
\newtheorem{corollary}[theorem]{Corollary}
\newtheorem{lemma}[theorem]{Lemma}

\newtheorem{conjecture}[theorem]{Conjecture}
\newtheorem{claim}[theorem]{Claim}

\theoremstyle{definition}

\AtAppendix{\counterwithin{theorem}{section}}

\titleformat{\section}[hang]{\scshape\large\bfseries\filcenter}{\S\thesection}{4pt}{}
\titleformat{\subsection}[hang]{\scshape\bfseries}{\thesubsection.}{4pt}{}

\allowdisplaybreaks

\newcommand{\tss}[1]{\textsuperscript{#1}}

\newcommand{\on}[1]{
	\operatorname{#1}
}

\newcommand\mder{{\Delta\!\!\!\!\!\hbox{\raisebox{0.3ex}{\tiny\ \textbullet}}}\,}

\newcommand{\tightoverset}[2]{
  \mathop{#2}\limits^{\vbox to -.5ex{\kern-1.15ex\hbox{$#1$}\vss}}}

\newcommand\restr[2]{{% we make the whole thing an ordinary symbol
  \left.\kern-\nulldelimiterspace % automatically resize the bar with \right
  #1 % the function
  \vphantom{\big|} % pretend it's a little taller at normal size
  \right|_{#2} % this is the delimiter
}}

\makeatletter
\newcommand{\subalign}[1]{%
  \vcenter{%
    \Let@ \restore@math@cr \default@tag
    \baselineskip\fontdimen10 \scriptfont\tw@
    \advance\baselineskip\fontdimen12 \scriptfont\tw@
    \lineskip\thr@@\fontdimen8 \scriptfont\thr@@
    \lineskiplimit\lineskip
    \ialign{\hfil$\m@th\scriptstyle##$&$\m@th\scriptstyle{}##$\hfil\crcr
      #1\crcr
    }%
  }%
}
\makeatother

\newcommand\blfootnote[1]{%
  \begingroup
  \renewcommand\thefootnote{}\footnote{#1}%
  \addtocounter{footnote}{-1}%
  \endgroup
}

\newcommand\ex{\mathop{\mathbb{E}}}

\newcommand{\exx}{
  \mathop{
    \mathchoice{\vcenter{\hbox{\larger[4]$\mathbb{E}$}}}
               {\kern0pt\mathbb{E}}
               {\kern0pt\mathbb{E}}
               {\kern0pt\mathbb{E}}
  }\displaylimits
}

\makeatletter
\newcommand*\bcdot{\mathpalette\bigcdot@{0.5}}
\newcommand*\bigcdot@[2]{\mathbin{\vcenter{\hbox{\scalebox{#2}{$\m@th#1\bullet$}}}}}
\makeatother

\makeatletter
\def\blfootnote{\gdef\@thefnmark{}\@footnotetext}
\makeatother

\newcommand\id{\mathbbm{1}}

\begin{document}
\begin{center}\Large\noindent{\bfseries{\scshape Approximately Symmetric Forms Far From Being Exactly Symmetric}}\\[24pt]\normalsize\noindent{\scshape Luka Mili\'cevi\'c\dag}\\[6pt]
\end{center}
\blfootnote{\noindent\dag\ Mathematical Institute of the Serbian Academy of Sciences and Arts\\\phantom{\dag\ }Email: luka.milicevic@turing.mi.sanu.ac.rs}

\footnotesize
\begin{changemargin}{1in}{1in}
\centerline{\sc{\textbf{Abstract}}}
\phantom{a}\hspace{12pt}~Let $V$ be a finite-dimensional vector space over $\mathbb{F}_p$. We say that a multilinear form $\alpha \colon V^k \to \mathbb{F}_p$ in $k$ variables is $d$-\emph{approximately symmetric} if the partition rank of difference $\alpha(x_1, \dots, x_k) - \alpha(x_{\pi(1)}, \dots, x_{\pi(k)})$ is at most $d$ for every permutation $\pi \in \on{Sym}_k$. In a work concerning the inverse theorem for the Gowers uniformity $\|\cdot\|_{\mathsf{U}^4}$ norm in the case of low characteristic, Tidor conjectured that any $d$-approximately symmetric multilinear form $\alpha \colon V^k \to \mathbb{F}_p$ differs from a symmetric multilinear form by a multilinear form of partition rank at most $O_{p,k,d}(1)$ and proved this conjecture in the case of trilinear forms. In this paper, somewhat surprisingly, we show that this conjecture is false. In fact, we show that approximately symmetric forms can be quite far from the symmetric ones, by constructing a multilinear form $\alpha \colon \mathbb{F}_2^n \times \mathbb{F}_2^n \times \mathbb{F}_2^n \times \mathbb{F}_2^n \to \mathbb{F}_2$ which is 3-approximately symmetric, while the difference between $\alpha$ and any symmetric multilinear form is of partition rank at least $\Omega(\sqrt[3]{n})$. 
\end{changemargin}
\normalsize
\section{Introduction}

\hspace{18pt}For a function $f \colon G \to \mathbb{C}$ on a finite abelian group $G$, Gowers uniformity norms $\|\cdot\|_{\mathsf{U}^k}$, introduced by Gowers in~\cite{TimSze} are given by the formula
\[\|f\|_{\mathsf{U}^k}^{2^k} = \exx_{a_1, \dots, a_k, x \in G} \mder_{a_1} \dots \mder_{a_k}f(x),\]
where $\mder_a$ stands for the discrete multiplicative derivative operator defined by $\mder_a f(x) = f(x + a) \overline{f(x)}$. Gowers introduced these norms in order to obtain a quantitative proof of Szemer\'edi's theorem on arithmetic progressions, and they serve as a measure of the higher order quasirandomness of functions defined on finite abelian groups. A basic illustration of this phenomenon is given by the following fact: whenever $A \subset \mathbb{Z}/N \mathbb{Z}$ is a set of size $\delta N$ such that $\|\id_A - \delta\|_{\mathsf{U}^k} = o(1)$, then $A$ has $(1 + o(1)) \delta^{k+1}N^2$ arithmetic progressions of length $k + 1$. This motivates the study of functions which have large uniformity norms. The results which describe such functions $f \colon G \to \mathbb{D} = \{z \in \mathbb{C} \colon |z| \leq 1\}$ are referred to as the \emph{inverse theorems for uniformity norms} and typically have the following form: given the group $G$ and order $k$, there is some algebraically structured family of functions $\mathcal{Q}$, depending on $G$ and $k$, such that whenever $f \colon G \to \mathbb{D}$ is a function with the norm bound $\|f\|_{\mathsf{U}^k} \geq c$, then one may find an obstruction function $q \in \mathcal{Q}$ such that $\Big|\ex_{x \in G} f(x) \overline{q(x)}\Big|\geq \Omega_{c,k}(1)$. One also requires that $\mathcal{Q}$ is roughly minimal in the sense that an approximate version of a converse holds; namely whenever $\Big|\ex_{x \in G} f(x) \overline{q(x)}\Big|\geq c$ holds for a function $f \colon G \to \mathbb{D}$ and obstruction $q \in \mathcal{Q}$, then we also have the norm bound $\|f\|_{\mathsf{U}^k} \geq \Omega_{c, k}(1)$.\\

The family of obstruction functions can be taken to be nilsequences (which we will not define here) when $G = \mathbb{Z}/N\mathbb{Z}$, as shown by Green, Tao and Ziegler~\cite{StrongUkZ}, and phases of non-classical polynomials (which we shall define later) when $G= \mathbb{F}_p^n$, which follows from results of Bergelson, Tao and Ziegler~\cite{BergelsonTaoZiegler} and Tao and Zielger~\cite{TaoZiegler}. Another approach to these questions is via the theory of nilspaces developed in papers by Szegedy~\cite{Szeg}, Camarena and Szegedy~\cite{CamSzeg}, and Candela, Gonz\'alez-S\'anchez and Szegedy~\cite{CandGonzSzeg}. (See also detailed treatments of this theory in~\cite{GMV1},~\cite{GMV2},~\cite{GMV3}.)\\ 

The inverse theorems mentioned above are either ineffective or give poor bounds and, given the applications, it is of interest to make the proofs quantitative. For general $k$, this was achieved by Manners~\cite{Manners} for the case when $G = \mathbb{Z}/N\mathbb{Z}$ and by Gowers and the author in~\cite{genPaper} for the case when $G = \mathbb{F}_p^n$, provided $p \geq k$, which is known as the \emph{high-characteristic case}, when the family of obstruction functions reduces to polynomials in the usual sense. Previously, quantitative bounds were obtained in the inverse question for $\|\cdot\|_{\mathsf{U}^3}$ norm by Green and Tao~\cite{StrongU3} for abelian groups of odd order and by Samorodnitsky when $G = \mathbb{F}_2^n$ in~\cite{SamorU3} (see also a very recent work of Jamneshan and Tao~\cite{U3AsgarTao}).\\

On the other hand, in the so-called \emph{low characteristic case}, where $p < k$, the bounds are still ineffective. However, even in that case~\cite{genPaper} gives a strong partial result. 

\begin{theorem}[Gowers and Mili\'cevi\'c~\cite{genPaper}]\label{partialInverseTheorem}Suppose that $f \colon \mathbb{F}^n_p \to \mathbb{D}$ is a function such that $\|f\|_{\mathsf{U}^{k}} \geq c$. Then there exists a multilinear form $\alpha \colon \underbrace{\mathbb{F}^n_p \times \mathbb{F}^n_p \times \dots \times \mathbb{F}^n_p}_{k-1} \to \mathbb{F}_p$ such that
\begin{equation}\label{partialCorrelationForm}\Big| \exx_{x, a_1, \dots, a_{k-1}} \mder_{a_1} \dots \mder_{a_{k-1}} f(x) \omega^{\alpha(a_1, \dots, a_{k-1})}\Big| \geq \Big(\exp^{(O_k(1))}(O_{k,p}(c^{-1}))\Big)^{-1}.\end{equation}
\end{theorem}

From now on, we focus on $G = \mathbb{F}_p^n$ in the rest of the introduction. Before proceeding with the discussion, we need to recall the notion of the partition rank of a multilinear form introduced by Naslund in~\cite{Naslund}. It is defined to be the least number $m$ such that a multilinear form $\alpha \colon G^d \to \mathbb{F}_p$ can be expressed as
\[\alpha(x_1, \dots, x_d) = \sum_{i \in [m]} \beta_i(x_j \colon j \in I_i) \gamma_i(x_j \colon j \in [d] \setminus I_i),\]
where $\beta_i \colon G^{I_i} \to \mathbb{F}_p$ and $\gamma_i \colon G^{[d] \setminus I_i} \to \mathbb{F}_p$ are multilinear maps for $i \in [m]$. We may think of partition rank as a measure of distance between two multilinear forms; the smaller the partition rank of their difference is, the closer they are. As an illustration of this principle in the context of Theorem~\ref{partialInverseTheorem}, we have the following lemma.

\begin{lemma}\label{closeformsreplacementinverse}Suppose that a function $f \colon G \to \mathbb{D}$ and a multilinear form $\alpha \colon G^{k-1} \to \mathbb{F}_p$ satisfy
\[\Big| \exx_{x, a_1, \dots, a_{k-1}} \mder_{a_1} \dots \mder_{a_{k-1}} f(x) \omega^{\alpha(a_1, \dots, a_{k-1})}\Big| \geq c.\]
Let $\beta \colon G^{k-1} \to \mathbb{F}_p$ be another multilinear form such that $\on{prank}(\alpha - \beta) \leq r$. Then
\[\Big| \exx_{x, a_1, \dots, a_{k-1}} \mder_{a_1} \dots \mder_{a_{k-1}} f(x) \omega^{\beta(a_1, \dots, a_{k-1})}\Big| \geq c p^{-2r}.\]
\end{lemma}

Returning to the discussion of the inverse theorems for uniformity norms in finite vector spaces, the deduction of the inverse theorem when $p \geq k$ proceeds by studying the multilinear form $\alpha$ provided by the Theorem~\ref{partialInverseTheorem}. The symmetry argument of Green and Tao~\cite{StrongU3} and the good bounds for the analytic versus partition rank problem~\cite{Janzer2},~\cite{LukaRank} show that $\alpha$ is \emph{$r$-approximately symmetric} for some $r = \exp^{(O_k(1))}(O_{k,p}(c^{-1}))$, by which we mean that the partition rank of the multilinear form $(x_1, \dots, x_{k-1}) \mapsto \alpha(x_1, \dots, x_{k-1}) - \alpha(x_{\pi(1)}, \dots, x_{\pi(k-1)})$ is at most $r$ for all $\pi \in \on{Sym}_{k-1}$. To finish the proof of the inverse theorem, at the very last step we invoke the assumption $p \geq k$, which allows us to define the symmetric multilinear map $\sigma \colon G^{k-1} \to \mathbb{F}_p$ by
\[\sigma(a_1, \dots, a_{k-1}) = \frac{1}{(k-1)!} \sum_{\pi \in \on{Sym}_{k-1}} \alpha(a_{\pi(1)}, \dots, a_{\pi(k-1)}),\]
which satisfies $\on{prank} (\sigma - \alpha) \leq (k-1)! r$ since $\alpha$ is $r$-approximately symmetric. Lemma~\ref{closeformsreplacementinverse} allows us to replace $\alpha$ by $\sigma$. As it turns out, when $p \geq k$ the polarization identity shows that all symmetric forms are iterated discrete additive derivatives\footnote{Similarly to discrete multiplicative derivative, for $a \in G$ we define discrete additive derivative operator $\Delta_a$ by expression $\Delta_a f(x) = f(x + a) - f(x)$ for a function $f \colon G \to H$ from $G$ to another abelian group $H$.} of polynomials, showing that the function $x \mapsto f(x) \omega^{\alpha(x,\dots, x)}$ has large $\|\cdot\|_{\mathsf{U}^{k-1}}$ norm, which completes the proof.\\

\noindent\textbf{Low characteristic obstacles.} Let us now define non-classical polynomials which are the relevant obstructions in the low characteristic case. A function $f \colon \mathbb{F}_p^n \to \mathbb{T} = \mathbb{R}/\mathbb{Z}$ is a non-classical polynomial of degree at most $d$ if $\Delta_{a_1} \dots \Delta_{a_{d+1}} f(x) = 0$ for all $a_1, \dots, a_{d+1}, x \in \mathbb{F}_p^n$. (See~\cite{TaoZiegler} for further details, including alternative description of non-classical polynomials.)\\
\indent The first obvious question is, given that the family of obstruction functions is richer when $p < k$ due to emergence of non-classical polynomials, how could we get from multilinear forms in~\eqref{partialCorrelationForm} to a non-classical polynomial? It turns out that, as in the case of classical polynomials, the iterated discrete additive derivative of a non-classical polynomial (applied the right number of times) is a symmetric multilinear form and it is possible to give characterizations of the forms that arise in this way. The following lemma of Tidor~\cite{Tidor}, building upon earlier work of Tao and Ziegler, achieves this goal.

\begin{lemma}[Tidor~\cite{Tidor}]Let $\alpha \colon G^{k-1} \to \mathbb{F}_p$ be a multilinear form. Then $\alpha$ is the discrete additive derivative of order $k-1$ of a non-classical polynomial of degree $k-1$ if and only if $\alpha$ is symmetric and
\begin{equation}\alpha(\underbrace{x, \dots, x}_p, y, a_{p+2}, \dots, a_{k-1}) = \alpha(x, \underbrace{y, \dots, y}_p, a_{p+2}, \dots, a_{k-1}).\label{discderssymmformeqn}\end{equation}
\end{lemma}

We say that a multilinear form is \emph{strongly symmetric} if it is symmetric and obeys the additional condition~\eqref{discderssymmformeqn}, as in the lemma above. Therefore, we may again pass from a multilinear form $\alpha$ to the desired obstruction, provided we can show some additional properties of $\alpha$.\\

In fact, in his work on the inverse question for $\|\cdot\|_{\mathsf{U}^4}$ norm in the case of low characteristic~\cite{Tidor}, Tidor first showed that one may assume that $\alpha$ is symmetric and then use that information to prove strong symmetry. Given all this, Tidor formulated the following conjecture, which he proved for the case of trilinear maps.

\begin{conjecture}[Tidor~\cite{Tidor}]Let $\alpha \colon G^k \to \mathbb{F}_p$ be an $r$-approximately symmetric multilinear form. Then there exists a symmetric multilinear form $\sigma \colon G^k \to \mathbb{F}_p$ such that $\on{prank}(\sigma - \alpha) \leq O(r^{O(1)})$ (where the implicit constants may depend on $p$ and $k$).\end{conjecture}

The main result of this paper is that, despite its formulation being rather natural, the conjecture above is false.

\begin{theorem}\label{mainCounterThm}Given a sufficiently large positive integer $n$ there exists a multilinear form $\alpha \colon \mathbb{F}_2^n \times \mathbb{F}_2^n \times \mathbb{F}_2^n \times \mathbb{F}_2^n \to \mathbb{F}_2$ which is 3-approximately symmetric and $\on{prank}(\sigma - \alpha) \geq \Omega(\sqrt[3]{n})$ for all symmetric multilinear forms $\sigma$.\end{theorem} 

Regarding the approach to a quantitative inverse theorem for uniformity norms in the case of low characteristic, this means that one needs to work more closely with the assumption~\eqref{partialCorrelationForm}. In fact, it is possible to overcome the additional difficulties identified in this paper, and to prove a quantitative inverse theorem for uniformity norms in the case of low characteristic, which we do in a forthcoming paper.

\begin{theorem}[Mili\'cevi\'c~\cite{LukaF2}]Suppose that $f \colon \mathbb{F}^n_2 \to \mathbb{D}$ be a function such that $\|f\|_{\mathsf{U}^{k}} \geq c$. Then there exists a non-classical polynomial $q \colon \mathbb{F}_2^n \to \mathbb{T}$ of degree at most $k-1$ such that 
\[\Big|\exx_{x} f(x) \exp\Big(2 \pi i q(x)\Big)\Big| \geq \Big(\exp^{(O_k(1))}(O_{k,p}(c^{-1}))\Big)^{-1}.\]
\end{theorem}
\vspace{\baselineskip}
\noindent\textbf{Counterexample overview.} The multilinear form $\alpha \colon  \mathbb{F}_2^n \times \mathbb{F}_2^n \times \mathbb{F}_2^n \times \mathbb{F}_2^n \to \mathbb{F}_2$ that will serve us as a counterexample will have the crucial properties that it is symmetric in the first three variable, while satisfying the identity
\[\alpha(x,y,z,w) + \alpha(w,y,z,x) = \rho(x,y)\rho(z,w) + \rho(x,z) \rho(y,w)\]
for some bilinear map $\rho$ of high rank\footnote{In the case of bilinear forms, the partition rank becomes just the usual notion of rank from linear algebra.}. It is easy to see that such a form $\alpha$ is necessarily 3-approximately symmetric. On the other hand, to prove that $\alpha$ is far from symmetric multilinear forms we use bilinear regularity method, used in~\cite{U4paper}, which consists of passing to subspaces where the rank of bilinear maps is large and then relying on the high-rank property to obtain equidistribution of values of the relevant bilinear maps. Using the usual graph-theoretic regularity method (i.e.\ Szemer\'edi regularity lemma and related tools) would give much worse bounds in Theorem~\ref{mainCounterThm}.\\

\noindent\textbf{Acknowledgements.} This work was supported by the Serbian Ministry of Education, Science and Technological Development through Mathematical Institute of the Serbian Academy of Sciences and Arts.

\section{Preliminaries}

\noindent For the rest of the paper, fix a positive integer $n$ and set $G = \mathbb{F}_2^n$. In this preliminary section we setup the notation for an action of the symmetry group $\on{Sym}_4$ on $G^4$, we recall the notion and properties of the rank of bilinear maps and we discuss the bilinear regularity method needed for the proof of Theorem~\ref{mainCounterThm}.\\

\noindent\textbf{Action of $\on{Sym}_4$.} We define a natural action of $\on{Sym}_4$ on $G^4$ given by permuting the coordinates, which is similar to the left regular representation of the group $\on{Sym}_4$. For a permutation $\pi \in \on{Sym}_4$ we misuse the notation and write $\pi \colon G^4 \to G^4$ for the map defined by $\pi(x_1, x_2, x_3, x_4) = (x_{\pi^{-1}(1)},$ $x_{\pi^{-1}(2)},$ $x_{\pi^{-1}(3)},$ $x_{\pi^{-1}(4)})$. This defines an action on $G^4$; indeed, given two permutations $\pi, \sigma \in \on{Sym}_4$ and $x_1, \dots, x_4 \in G$ we have
\begin{align*}\sigma(\pi(x_1, x_2, x_3, x_4)) = &\sigma(x_{\pi^{-1}(1)}, x_{\pi^{-1}(2)}, x_{\pi^{-1}(3)}, x_{\pi^{-1}(4)})\\
 = &(x_{\pi^{-1}(\sigma^{-1}(1))}, x_{\pi^{-1}(\sigma^{-1}(2))}, x_{\pi^{-1}(\sigma^{-1}(3))}, x_{\pi^{-1}(\sigma^{-1}(4))})\\
 = &(x_{(\sigma \pi)^{-1}(1)}, x_{(\sigma \pi)^{-1}(2)}, x_{(\sigma \pi)^{-1}(3)}, x_{(\sigma \pi)^{-1}(4)}) \\
= &(\sigma \pi)(x_1, \dots, x_4).\end{align*}

Hence, given a multilinear form $\alpha \colon G^4 \to \mathbb{F}_2$ and a permutation $\pi$ inducing the map $\pi \colon G^4 \to G^4$, we may compose the two maps and the composition $\alpha \circ \pi$ would also be a multilinear form. For example if $\pi = (1\,\,2\,\,3)$ in the cycle notation, then $\alpha \circ \pi(x_1, x_2, x_3, x_4) = \alpha(x_3, x_1, x_2, x_4)$. With this notation, our main result can be expressed as follows: there exists a multilinear form $\alpha \colon G^4 \to \mathbb{F}_2$ such that $\alpha + \alpha \circ \pi$ is of low partition rank for all $\pi \in \on{Sym}_4$ and $\alpha$ differs from any symmetric multilinear form by a multilinear form of large partition rank.\\

\noindent\textbf{Rank of bilinear maps.} Let $U, V$ be finite-dimensional vector spaces over $\mathbb{F}_2$. Let $\beta \colon U \times V \to \mathbb{F}_2$ be a bilinear form. Fix a scalar product on $V$ and let $B \colon U \to V$ be the map such that $\beta(x,y) = B(x) \cdot y$ for all $x \in U$ and $y \in V$. We define the rank of $\beta$ to be the rank of $B$. The following lemma gives a few other characterizations of the rank and shows that the rank is well-defined (this follows from part \textbf{(i)}).

\begin{lemma}[Alternative characterizations of rank]\label{rankscharlemma}Let $\beta \colon U \times V \to \mathbb{F}_2$ be a bilinear form. 
\begin{itemize}
\item[\textbf{(i)}]We have $\ex_{x \in U,y \in V} (-1)^{\beta(x,y)} = 2^{-\on{rank} \beta}$. 
\item[\textbf{(ii)}]Whenever $\beta(x,y) = \sum_{i \in [s]} u_i(x) v_i(y)$ for linear forms $u_1, \dots, u_s \colon U \to \mathbb{F}_2$ and $v_1, \dots, v_s \colon V \to \mathbb{F}_2$ such that $u_1, \dots, u_s$ are linearly independent and $v_1, \dots, v_s$ are linearly independent, then $s = \on{rank} \beta$.
\item[\textbf{(iii)}]The rank is the least number $s$ such that  $\beta(x,y) = \sum_{i \in [s]} u_i(x) v_i(y)$ for linear forms $u_1, \dots, u_s \colon U \to \mathbb{F}_2$ and $v_1, \dots, v_s \colon V \to \mathbb{F}_2$.
\end{itemize}\end{lemma}

\begin{proof}\textbf{Proof of (i).} If $\beta(x,y) = B(x) \cdot y$ for all $x \in U, y \in V$, then
\[\exx_{x \in U,y \in V} (-1)^{\beta(x,y)} = \exx_{x \in U, y \in V} (-1)^{B(x) \cdot y} = \exx_{x \in U} \id(B(x) = 0) = \exx_{x \in U} \id(x \in \ker B) = |\ker B| / |U|.\]
By the rank-nullity theorem $|\ker B|/|U| = 1/|\on{Im} B|$. The rank of $\beta$ is defined as the dimension of the image space $\on{Im} B$, hence 
\[\exx_{x \in U,y \in V} (-1)^{\beta(x,y)} = 2^{-\on{rank}\beta},\]
as desired.\\

\noindent\textbf{Proof of (ii).} We claim that the linear map $u \colon U \to \mathbb{F}_2^s$ is surjective. If it was not surjective, the subspace $\on{Im} u \leq \mathbb{F}_2^s$ would be proper and there would exist a non-zero vector $\lambda \in (\on{Im} u)^\perp$. But then $\lambda_1 u_1 + \dots + \lambda_s u_s$ would be identically zero, which would be in contradiction with the assumption that $u_1, \dots, u_s$ are independent. In particular, the linear map $u \colon U \to \mathbb{F}_2^s$ takes all values in $\mathbb{F}_2^s$ an equal number of times (namely $|\ker u| = p^{-s}|U|$ times). Using the same property for $v_1, \dots, v_s$ and claim \textbf{(i)} of the lemma we conclude that
\[2^{-\on{rank} \beta} = \exx_{x \in U, y \in V} (-1)^{\beta(x,y)} = \exx_{x \in U, y \in V} (-1)^{\sum_{i \in [s]} u_i(x) v_i(y)} = \exx_{a, b \in \mathbb{F}_2^s} (-1)^{a \cdot b} = \exx_{a \in \mathbb{F}^s_2} \id(a = 0) = 2^{-s},\]
proving that $s = \on{rank} \beta$.\\

\noindent\textbf{Proof of (iii).} Let $d = \dim U$. Clearly, such a decomposition exists, as we may simply take a basis $e_1, \dots, e_d$ of $U$, giving coordinates $x_1, \dots, x_d$ of vectors $x \in U$, and consider 
\[\beta(x,y) = \sum_{i \in [d]} x_i \beta(e_i, y).\]
On the other hand, if we have a decomposition with smallest possible $s$ then $u_1, \dots, u_s$ need to be linearly independent. To see that, note that if have linear dependence, then (after a possible reordering of the forms) we have $u_1 = \sum_{j \in [2,s]} \mu_j u_j$, so 
\[\beta(x,y) = \sum_{j \in [2,s]} u_j(x) (v_j(y) + \mu_j v_1(y)),\]
which has $s-1$ terms in the sum, which is a contradiction. Hence $u_1, \dots, u_s$ are linearly independent, and analogously, so are $v_1, \dots, v_s$. The claim follows from the part \textbf{(ii)}.\end{proof}

Let us also record two very simple but useful facts about bilinear forms of low rank.

\begin{lemma}\label{lowranktozero}Let $\beta \colon U \times V \to \mathbb{F}_2$ be a bilinear form of rank $r$. Then there exists a subspace $U' \leq U$ of codimension at most $r$ in $U$ such that $\beta|_{U' \times V} = 0$.\end{lemma}

\begin{proof}By the definition of rank, $r$ is the rank of the linear map $B \colon U \to V$ that satisfies $\beta(x,y) = B(x) \cdot y$ for all $x \in U, y \in V$ for a given scalar product on $V$. By the rank-nullity theorem, the kernel $U' = \on{ker} B \leq U$ has codimension $r$. Hence, when $x \in U', y \in V$ we have $\beta(x,y) = B(x) \cdot y = 0$, as desired.\end{proof}

\begin{lemma}\label{rankdecsub}Let $\beta \colon U \times V \to \mathbb{F}_2$ be a bilinear form of rank $r$. Let $U' \leq U$ be a subspace of codimension $d$ inside $U$. Then $\beta|_{U' \times V}$ has rank at least $r -d$.\end{lemma}

\begin{proof}Let $s$ be the rank of $\beta|_{U' \times V}$. By part \textbf{(iii)} of Lemma~\ref{rankscharlemma} we have linear forms $u_1, \dots, u_s \colon U' \to \mathbb{F}_2$ and $v_1, \dots, v_s \colon V \to \mathbb{F}_2$ such that $\beta(x,y) = \sum_{i \in [s]} u_i(x) v_i(y)$ holds for all $x \in U'$ and $y \in V$. We may extend each $u_i$ to a linear form $\tilde{u}_i \colon U \to \mathbb{F}_2$. Let $\beta'(x,y) = \beta(x,y) + \sum_{i \in [s]} \tilde{u}_i(x) v_i(y)$. The map $\beta'$ is a bilinear form on $U \times V$ which vanishes on $U' \times V$. We claim that $\beta'$ has rank at most $d$. Since we also have that $\beta(x,y) = \beta'(x,y) + \sum_{i \in [s]} \tilde{u}_i(x) v_i(y)$, it follows that $\beta$ can be written as a sum of at most $s+d$ terms of the form $u'(x) v'(y)$ for suitable linear forms $u' \colon U \to \mathbb{F}_2$ and $v' \colon V \to \mathbb{F}_2$ and so by part \textbf{(iii)} of Lemma~\ref{rankscharlemma} we have $r \leq s + d$ as desired. We now return to showing that $\on{rank} \beta' \leq d$.\\
\indent Since $U'$ has codimension $d$ inside $U$, we may find linearly independent elements $e_1, \dots, e_d \in U$ such that $U = \langle e_1, \dots, e_d \rangle \oplus U'$. We thus obtain linear forms $\varphi_1, \dots, \varphi_d \colon U \to \mathbb{F}_2$ and a linear map $\pi \colon U \to U'$ such that for each $x \in U$ we have $x = \sum_{i \in [d]} \varphi_i(x) e_i + \pi(x)$. Using this decomposition, for arbitrary $x \in U, y \in V$ we see that 
\[ \beta'(x,y) = \beta'\Big(\sum_{i \in [d]} \varphi_i(x) e_i + \pi(x), y\Big) = \sum_{i \in [d]} \varphi_i(x) \beta'(e_i, y) + \beta'(\pi(x), y) = \sum_{i \in [d]} \varphi_i(x) \beta'(e_i, y).\]
Part \textbf{(iii)} of Lemma~\ref{rankscharlemma} implies that $\on{rank}\beta' \leq d$.\end{proof}

\noindent\textbf{Bilinear regularity method.} In the proof that our example has the desired properties we need the algebraic regularity method for bilinear maps. This method was used in~\cite{U4paper}. The following lemma, in the spirit of Corollary 5.2 of~\cite{U4paper}, essentially shows that for a given bilinear map we may pass to a subspace on which it behaves quasirandomly (which in the bilinear setting simply means that the restriction of the bilinear map has high rank). 

\begin{lemma}[Bilinear regularity lemma]\label{bilRegularityLemma}Let $m \geq 1$ be a positive integer. Let $U$ be a finite-dimensional vector space over $\mathbb{F}_2$ and let $\rho, \beta_1, \dots, \beta_r \colon U \times U \to \mathbb{F}_2$ be bilinear forms such that $\on{rank} \rho \geq (4r + 1)m$. Then there exist a subspace $U' \leq U$ of codimension at most $2rm$ and bilinear forms $\alpha_1, \dots, \alpha_s \colon U' \times U' \to \mathbb{F}_2$, where $s \leq r$, such that every non-zero linear combination of $\alpha_1, \dots, \alpha_s$ and $\rho|_{U' \times U'}$ has rank at least $m$, while every bilinear map among $\beta_1, \dots, \beta_r$ equals a linear combination of $\alpha_1, \dots, \alpha_s$ and $\rho$ on $U' \times U'$.\end{lemma}

\begin{proof}Let us first set $s = r$, $U' = U$ and $\alpha_i = \beta_i$. We shall modify the number $s$, subspace $U'$ and maps $\alpha_1, \dots, \alpha_s$ throughout the proof. At each step of the proof the number of forms $s$ will decrease by 1, while the dimension of $U'$ will increase by at most $2m$. Note that every bilinear map among $\beta_1, \dots, \beta_r$ equals some $\alpha_i$ on $U' \times U'$ so we just need to make sure that ranks of all non-zero linear combinations of $\alpha_i$ and $\rho$ are sufficiently large. Suppose on the contrary that there is a linear combination of maps $\alpha_1, \dots, \alpha_s$ and $\rho|_{U' \times U'}$ such that
\[\on{rank} \Big(\lambda_1 \alpha_1 + \dots + \lambda_s \alpha_s + \mu \rho|_{U' \times U'}\Big) < m.\]
Since by Lemma~\ref{rankdecsub} we have $\on{rank} \rho|_{U' \times U'} \geq \on{rank} \rho - 2(\dim U - \dim U') \geq m$, there exists a non-zero $\lambda_j$. Reordering $\alpha_i$ if necessary, we may assume that $\lambda_s = 1$. Thus, by Lemma~\ref{rankscharlemma}\textbf{(iii)} there exist linear forms $v_1, \dots, v_m, v'_1, \dots, v'_m \colon U' \to \mathbb{F}_2$ such that 
\[\alpha_s(x,y) = \lambda_1 \alpha_1(x,y) + \dots + \lambda_{s-1} \alpha_{s-1}(x,y) + \mu \rho(x,y) + v_1(x)v'_1(y)+\dots + v_m(x)v'_m(y).\]
We may replace $U'$ by $U' \cap \{x \in U' \colon (\forall i \in [m]) v_i(x) = v'_i(x) = 0\}$ and remove the form $\alpha_s$ from the sequence. The property that every bilinear map among $\beta_1, \dots, \beta_r$ equals a suitable linear combination on $U' \times U'$ is preserved and the procedure must terminate after at most $r$ steps.\end{proof}

The next lemma is an algebraic counting lemma, similar to Lemma 5.3 of~\cite{U4paper}.

\begin{lemma}[Bilinear counting lemma]\label{bilCountingLemma}Suppose that $\alpha \colon \mathbb{F}_2^n \times \mathbb{F}_2^n \to \mathbb{F}_2^r$ is a bilinear map with the property that the form $\lambda \cdot \alpha$ has rank at least $m$ for all non-zero vectors $\lambda \in \mathbb{F}_2^r$. Let $C$ be a coset of a subspace of $\mathbb{F}_2^n$ of codimension $d$. Let $\varepsilon >0 $. If $m > 4r + 8d + 4 \log_2\varepsilon^{-1}$ then $\alpha|_{C \times C}$ takes every value in $\mathbb{F}_2^r$ at least $(1 - \varepsilon)2^{-r}|C|^2$ times. In particular, if $m > 4r + 8d + 4$ then $\alpha|_{C \times C}$ is surjective.\end{lemma}

The reason for calling this lemma an algebraic counting lemma is that it is closely related to the more traditional counting results. For example, we can easily deduce that for any $\nu_1, \nu_2, \nu_3 \in \mathbb{F}_2^r$ the number of triples $x,y,z \in \mathbb{F}_2^n$ such that $\alpha(x,y) = \nu_1$, $\alpha(y,z) = \nu_2$ and $\alpha(z, x) = \nu_3$ is approximately $2^{-3r} \cdot 2^{3n}$ as long as $\alpha$ is sufficiently quasirandom. In the graph-theoretic language, this corresponds to counting triangles. Similarly, Lemma~\ref{bilRegularityLemma} is related to Szemer\'edi regularity lemma.

\begin{proof}Suppose that the restriction $\alpha|_{C \times C}$ takes the value $u \in \mathbb{F}_2^r$ at most $(1 - \varepsilon)2^{-r}|C|^2$ times. Then we have
\begin{align*}(1 - \varepsilon)2^{-r-2d} \geq & \exx_{x,y \in \mathbb{F}_2^n} \id_C(x)\id_C(y) \id(\alpha(x,y) = u) = \exx_{x,y \in \mathbb{F}_2^n} \id_C(x)\id_C(y) \exx_{\lambda \in \mathbb{F}_2^r} (-1)^{\lambda \cdot (\alpha(x,y) - u)}\\
= & 2^{-r} \exx_{x,y \in \mathbb{F}_2^n} \id_C(x)\id_C(y) + 2^{-r} \sum_{\lambda \in \mathbb{F}_2^r \setminus \{0\}} \exx_{x,y \in \mathbb{F}_2^n} \id_C(x)\id_C(y) (-1)^{\lambda \cdot (\alpha(x,y) - u)}\\
= & 2^{-r - 2d} + 2^{-r} \sum_{\lambda \in \mathbb{F}_2^r \setminus \{0\}} (-1)^{- \lambda \cdot u}\exx_{x,y \in \mathbb{F}_2^n} \id_C(x)\id_C(y) (-1)^{\lambda \cdot \alpha(x,y)}.\end{align*}
Using the triangle inequality, we see that
\[2^{-r} \sum_{\lambda \in \mathbb{F}_2^r \setminus \{0\}} \bigg|\exx_{x,y \in \mathbb{F}_2^n} \id_C(x)\id_C(y) (-1)^{\lambda \cdot \alpha(x,y)}\bigg| \geq \varepsilon 2^{-r-2d},\]
so by averaging we obtain a non-zero $\lambda$ such that 
\[\bigg|\exx_{x,y \in \mathbb{F}_2^n} \id_C(x)\id_C(y) (-1)^{\lambda \cdot \alpha(x,y)}\bigg| \geq \varepsilon 2^{-r-2d}.\]
By the two-dimensional Gowers-Cauchy-Schwarz inequality, the left hand side can be bounded above by the box norm of $(-1)^{\lambda \cdot \alpha}$, so 
\[\varepsilon^4 2^{-4r-8d} \leq \Big\|(-1)^{\lambda \cdot \alpha}\Big\|_{\square}^4 = \exx_{x,y \in \mathbb{F}_2^n} (-1)^{\lambda \cdot \alpha(x,y)},\]
which equals $2^{-\on{rank}(\lambda \cdot \alpha)}$ by Lemma~\ref{rankscharlemma}\textbf{(i)}. Hence $\varepsilon^4 2^{-4r-8d} \leq 2^{-m}$ so $m \leq 4r + 8d + 4 \log_2\varepsilon^{-1}$, which is a contradiction.\end{proof}

\section{Example}
Recall that $n$ is a fixed positive integer, which we think of as large, and that $G = \mathbb{F}_2^n$. Let $e_1, \dots, e_n$ be the standard basis of $G$. Our example will be the multilinear form $\phi \colon G^4 \to \mathbb{F}_2$ defined as
\[\phi(x, y, z, w) = \sum_{1 \leq i < j \leq n} \Big(x_i y_j z_j w_i + x_j y_i z_j w_i + x_j y_j z_i w_i\Big),\]
where the coordinates of vectors are taken with the respect to the fixed basis $e_1, \dots, e_n$. We first show that $\phi$ is approximately symmetric.

\begin{lemma}The multilinear form $\phi$ satisfies $\phi = \phi \circ (1\,\,2)$, $\phi = \phi \circ (1\,\,3)$ and
\[\phi(x, y, z, w) = \phi \circ (1\,\,4)(x, y, z, w) + \rho(x,y) \rho(z,w) + \rho(x,z) \rho(y,w)\]
where
\[\rho(x,y) = \sum_{i \in [n]} x_i y_i.\]\end{lemma}

\begin{proof}For each $i$ and $j$ the expression $x_i y_j z_j w_i + x_j y_i z_j w_i + x_j y_j z_i w_i$ is symmetric in $x,y$ and $z$ and hence the equalities $\phi = \phi \circ (1\,\,2)$ and $\phi = \phi \circ (1\,\,3)$ follow immediately. For the remaining transposition $(1\,\,4)$ we perform some algebraic manipulation
\begin{align*}\phi(x, y, z, w) + \phi(w, y, z, x) =& \sum_{1 \leq i < j \leq n} \Big((x_i y_j z_j w_i + x_j y_i z_j w_i + x_j y_j z_i w_i) + (w_i y_j z_j x_i + w_j y_i z_j x_i + w_j y_j z_i x_i)\Big) \\
= &\sum_{1 \leq i < j \leq n} \Big(x_j y_i z_j w_i + x_j y_j z_i w_i + x_i y_i z_j w_j + x_i y_j z_i w_j\Big)\\
= &\sum_{1 \leq i < j \leq n} \Big(x_j y_i z_j w_i +  x_i y_j z_i w_j\Big) +  \sum_{1 \leq i < j \leq n} \Big(x_j y_j z_i w_i + x_i y_i z_j w_j\Big)\\
= &\Big(\sum_{i,j \in [n] \colon i \not= j} x_j y_i z_j w_i \Big) + \Big(\sum_{i,j \in [n] \colon i \not= j} x_j y_j z_i w_i \Big)\\
= &\Big(\sum_{i,j \in [n] \colon i \not= j} x_j y_i z_j w_i \Big) + \Big(\sum_{i,j \in [n] \colon i \not= j} x_j y_j z_i w_i \Big) + 2 \cdot \Big(\sum_{i \in [n]} x_i y_i z_i w_i\Big)\\
= &\Big(\sum_{i,j \in [n]} x_j y_i z_j w_i \Big) + \Big(\sum_{i,j \in [n]} x_j y_j z_i w_i \Big)\\
= &\Big(\sum_{j \in [n]} x_j z_j \Big) \Big(\sum_{i \in [n]} y_i w_i \Big) + \Big(\sum_{j \in [n]} x_j y_j \Big) \Big(\sum_{i \in [n]} z_i w_i \Big)\\
= & \rho(x,z) \rho(y,w) + \rho(x,y) \rho(z,w),\end{align*}
as desired.\end{proof}

Since the transpositions $(1\,\,2), (1\,\,3)$ and $(1\,\,4)$ generate the whole symmetric group $\on{Sym}_4$ we immediately deduce that $\phi$ is 3-approximately symmetric.

\begin{corollary}For any permutation $\pi \in \on{Sym}_4$ we have $\on{prank} \Big(\phi + \phi \circ \pi\Big) \leq 3$.\end{corollary}

\begin{proof}Let $V$ be the vector space consisting of the multilinear forms on $G^4$ of the shape $\lambda_1 \rho(x,y) \rho(z,w) + \lambda_2\rho(x,z) \rho(y,w) + \lambda_3 \rho(x,w) \rho(y,z)$ for some scalars $\lambda_1, \lambda_2, \lambda_3 \in \mathbb{F}_2$. Since $\rho$ is a symmetric bilinear form it follows that $V$ is invariant under the action of $\on{Sym}_4$. The lemma above shows in particular that $\phi + \phi \circ \pi \in V$ whenever $\pi$ is one of the transpositions $(1\,\,2), (1\,\,3)$ and $(1\,\,4)$. Let $\pi \in \on{Sym}_4$ now be an arbitrary permutation. The transpositions $(1\,\,2), (1\,\,3)$ and $(1\,\,4)$ generate $\on{Sym}_4$ so we can write $\pi = \tau_1 \circ \tau_2 \circ \dots \circ \tau_r$ for some $\tau_1, \dots, \tau_r \in \{(1\,\,2), (1\,\,3), (1\,\,4)\}$. Then we have
\begin{align*}\phi \circ \pi + \phi = &\sum_{i \in [r]}\Big(\phi \circ \tau_i \circ \tau_{i+1} \circ \dots \circ \tau_r + \phi \circ \tau_{i+1} \circ \tau_{i+2} \circ \dots \circ \tau_r\Big)\\ 
= &\sum_{i \in [r]} \Big(\phi \circ \tau_i + \phi\Big) \circ \tau_{i+1} \circ \dots \circ \tau_r\end{align*}
which is a sum of $r$ forms, each of which is a member of $V$. Hence  $\phi + \phi \circ \pi \in V$ for all $\pi \in \on{Sym}_4$. Since the partition rank of forms in $V$ is at most 3, the proof is complete.\end{proof}

In the rest of this section, we show that the map $\phi$ is necessarily far from any symmetric multilinear form. 

\begin{theorem}Let $r \geq 1$ be a positive integer. Assume that $n \geq (1000r)^3$. For any symmetric multilinear form $\sigma \colon G^4 \to \mathbb{F}_2$ we have $\on{prank}(\phi - \sigma) > r$.
\end{theorem}

\begin{proof}Let $\sigma \colon G^4 \to \mathbb{F}_2$ be a symmetric multilinear form. Suppose on the contrary that $\on{prank}(\phi - \sigma) \leq r$. Then we may find linear forms $u^1_1, \dots,$ $u^1_r,$ $\dots,$ $u^4_1, \dots, u^4_r$, bilinear forms $\gamma^1_1, \dots,$ $\gamma^1_r, \dots,$ $\gamma^3_1, \dots, \gamma^3_r$ and $\tilde{\gamma}^1_1, \dots,$ $\tilde{\gamma}^1_r, \dots,$ $\tilde{\gamma}^3_1, \dots, \tilde{\gamma}^3_r$, and trilinear forms $\theta^1_1, \dots,$ $\theta^1_r, \dots,$ $\theta^4_1, \dots, \theta^4_r$ (we set some of the forms to be 0 to have a single parameter $r$ instead of a separate count for each sequence of forms) on the space $G$ such that 
\begin{align*}\phi(x,y,z,w) = \sigma(x,y,z,w) + &\sum_{i \in [r]} u^1_i(x) \theta^1_i(y,z,w) + \sum_{i \in [r]} u^2_i(y) \theta^2_i(x,z,w)\\
&\hspace{5cm}  + \sum_{i \in [r]} u^3_i(z) \theta^3_i(x,y,w)  + \sum_{i \in [r]} u^4_i(w) \theta^4_i(x, y,z)\\
+& \sum_{i \in [r]}\gamma^1_i(x,y) \tilde{\gamma}^1_i(z,w) + \sum_{i \in [r]}\gamma^2_i(x,z) \tilde{\gamma}^2_i(y,w) + \sum_{i \in [r]}\gamma^3_i(x,w) \tilde{\gamma}^3_i(y,z)\end{align*}
holds for all $x,y,z,w \in G$. Let us pass to the subspace $U = \{x \in G \colon (\forall d \in [4])(\forall i \in [r]) u^d_i(x) = 0\}$ which has codimension at most $4r$. When $x,y,z,w \in U$ then 
\[\phi(x,y,z,w) = \sigma(x,y,z,w) + \sum_{i \in [r]}\gamma^1_i(x,y) \tilde{\gamma}^1_i(z,w) + \sum_{i \in [r]}\gamma^2_i(x,z) \tilde{\gamma}^2_i(y,w) + \sum_{i \in [r]}\gamma^3_i(x,w) \tilde{\gamma}^3_i(y,z).\]

In the rest of proof we deal with bilinear forms primarily so we simply use the word rank instead of partition rank as the usual notion of rank is equivalent to the partition rank, as remarked earlier. We now gather all $6r$ bilinear forms $\gamma^1_1, \dots, \tilde{\gamma}^3_r$ above into a single sequence and apply the bilinear regularity lemma. Let $m \geq 1$ be an integer to be chosen later. Assume that $n \geq 24rm + 8r + m$ so that $\on{rank} \rho|_{U \times U} \geq 24rm + m$ (recall that $\on{rank} \rho = n$). Lemma~\ref{bilRegularityLemma} allows us to find a subspace $U' \leq U$ of codimension at most $12rm$ inside $U$, a positive integer $s \leq 6r$, bilinear forms $\alpha_1, \dots, \alpha_s \colon U' \times U' \to \mathbb{F}_2$ such that every map among $\gamma^1_1, \dots, \tilde{\gamma}^3_r$ equals a linear combination of $\alpha_1, \dots, \alpha_s$ and $\rho|_{U' \times U'}$ on $U' \times U'$ and all non-zero linear combinations of $\alpha_1, \dots, \alpha_s$ and $\rho|_{U' \times U'}$ have rank at least $m$.\\ 
\indent Expressing maps $\gamma^1_1, \dots, \tilde{\gamma}^3_r$ in terms of these bilinear forms, we conclude that whenever $x,y,z,w \in U'$
\begin{align}\phi(x,y,z,w) = \sigma(x,y,z,w) + \sum_{i,j \in [s]} \lambda_{ij} \alpha_i(x,y) \alpha_j(z,w) + \sum_{i,j \in [s]} \lambda'_{ij} \alpha_i(x,z) \alpha_j(y,w) + \sum_{i,j \in [s]} \lambda''_{ij} \alpha_i(x,w) \alpha_j(y,z)\nonumber\\
+ \rho(x,y) \beta_1(z,w) + \rho(x,z) \beta_2(y,w) + \rho(x,w) \beta_3(y,z) + \rho(y,z) \beta_4(x,w) + \rho(y,w) \beta_5(x,z) + \rho(z,w) \beta_6(x,y)\label{PhiDecomposition1}\end{align}
where $\lambda_{ij}, \lambda'_{ij}, \lambda''_{ij} \in \mathbb{F}_2$ are suitable coefficients and $\beta_1, \dots, \beta_6 \colon U' \times U' \to \mathbb{F}_2$ are suitable bilinear forms.\\

We now use the approximate symmetry properties of $\phi$ in order to deduce that the coefficients $\lambda_{ij}, \lambda'_{ij}, \lambda''_{ij}$ are symmetric in $i$ and $j$ and that forms $\alpha_i$ are essentially symmetric. This will eventually allows us to simplify the expression in~\eqref{PhiDecomposition1}. The identity $\phi + \phi \circ (1\,\,2) = 0$ will be used to prove the following claim.

\begin{claim}\label{partialsymmetryClaim1} Assume $m \geq 40(s + 1)$.
\begin{itemize}
\item[\textbf{(i)}] There exists a subspace $V^1 \leq U'$ such that $\dim V^1 \geq \dim U' - 2s^2 - 5s$ and for each $j \in [s]$ the bilinear form $\sum_{i \in [s]}\lambda_{ij} \alpha_i$ is symmetric on $V^1 \times V^1$.
\item[\textbf{(ii)}] For all $i,j \in [s]$, $\lambda'_{ij} = \lambda''_{ji}$.
\end{itemize}
\end{claim}
\begin{proof} Using the fact that $\phi + \phi \circ (1\,\,2) = 0$ and~\eqref{PhiDecomposition1} we get
\begin{align}0 = &\sum_{i,j \in [s]} \lambda_{ij} (\alpha_i(x,y) + \alpha_i(y,x)) \alpha_j(z,w) + \sum_{i,j \in [s]} (\lambda'_{ij}+ \lambda''_{ji}) \alpha_i(x,z) \alpha_j(y,w) + \sum_{i,j \in [s]} (\lambda''_{ij} + \lambda'_{ji})\alpha_i(x,w) \alpha_j(y,z)\nonumber\\
&\hspace{2cm}+\rho(x,z) (\beta_2(y,w) + \beta_4(y,w)) + \rho(x,w) (\beta_3(y,z) + \beta_5(y,z)) + \rho(y,z) (\beta_4(x,w) + \beta_2(x,w))\nonumber\\
&\hspace{4cm} + \rho(y,w) (\beta_5(x,z) + \beta_3(x,z)) + \rho(z,w) (\beta_6(x,y) + \beta_6(y,x))\label{12decompositionsymmID}\end{align}
for all $x,y,z,w \in U'$.\\

\noindent\textbf{Proof of (i).} Let $j \in [s]$ be given. Since any non-zero linear combination of $\rho|_{U' \times U'}, \alpha_1, \dots, \alpha_s$ has rank at least $m > 4(s + 2)$ by Lemma~\ref{bilCountingLemma} we can find $z,w \in U'$ such that $\rho(z,w) = 0$, $\alpha_i(z,w) = 0$ for all $i \not= j$ and $\alpha_j(z,w) = 1$. Using this choice of $z,w$ in~\eqref{12decompositionsymmID} we obtain
\begin{align*}&\sum_{i \in [s]}\lambda_{ij} (\alpha_i(x,y) + \alpha_i(y,x)) = \sum_{i \in [s]}\alpha_i(x,z) \Big(\sum_{\ell \in [s]} (\lambda'_{i \ell}+ \lambda''_{\ell i})  \alpha_\ell(y,w)\Big) + \sum_{i \in [s]} \alpha_i(x,w)  \Big(\sum_{\ell \in [s]} (\lambda''_{i\ell} + \lambda'_{\ell i})\alpha_\ell(y,z)\Big)\\
&\hspace{2cm}+\rho(x,z) (\beta_2(y,w) + \beta_4(y,w)) + \rho(x,w) (\beta_3(y,z) + \beta_5(y,z)) + \rho(y,z) (\beta_4(x,w) + \beta_2(x,w))\\
&\hspace{4cm} + \rho(y,w) (\beta_5(x,z) + \beta_3(x,z)) + \rho(z,w) (\beta_6(x,y) + \beta_6(y,x))\end{align*}
for all $x,y \in U'$. By Lemma~\ref{rankscharlemma}\textbf{(iii)} we conclude that $\sum_{i \in [s]}\lambda_{ij} (\alpha_i(x,y) + \alpha_i(y,x))$ has rank at most $2s + 5$, so by Lemma~\ref{lowranktozero} we may find a subspace $U'_j \leq U'$ with $\dim U'_j \geq \dim U' - 2s-5$ such that $\sum_{i \in [s]}\lambda_{ij} (\alpha_i(x,y) + \alpha_i(y,x)) = 0$ when $x,y \in U'_j$. We may take $V^1 = \cap_{j \in [s]} U'_j$.\\

\noindent\textbf{Proof of (ii).} Let $i, j \in [s]$ be given. This time we find elements $x,y,z,w \in U'$ such that $\alpha_1, \dots, \alpha_s$ and $\rho$ are equal to 0 at all 6 points in the set $\{(x,y), (x,z), (x,w), (y,z), (y,w), (z,w)\}$, with the two exceptions being $\alpha_i(x,z) = \alpha_j(y,w) = 1$. Once we have such elements $x,y,z,w$ we use them in~\eqref{12decompositionsymmID} which reduces to just $\lambda'_{ij}+ \lambda''_{ji} = 0$, proving that $\lambda'_{ij} = \lambda''_{ji}$, which is what we are after. To obtain such a quadruple of elements $x,y,z,w$, provided $m > 4(s + 2)$, we first apply Lemma~\ref{bilCountingLemma} to obtain $(x,z) \in U' \times U'$ such that $\rho(x,z) = 0$, $\alpha_\ell(x,z) = 0$ when $\ell \not= i$ and $\alpha_i(x,z) = 1$. Then we define a subspace
\[\tilde{U} = \{u \in U' \colon \rho(x,u) = \rho(z,u) = 0 \land (\forall i \in [s]) \alpha_i(x,u) = \alpha_i(u, x) = \alpha_i(z,u) = \alpha_i(u, z) = 0\}.\]
Provided $m > 4(s + 2) + 8(4s + 2)$, we may use Lemma~\ref{bilCountingLemma} one more time to find $(y,w) \in \tilde{U} \times \tilde{U}$ such that $\rho(y,w) = 0$, $\alpha_\ell(y,w) = 0$ when $\ell \not= j$ and $\alpha_j(y,w) = 1$, completing the proof.\end{proof}

Next, we use the second symmetry condition $\phi + \phi \circ (1\,\,3) = 0$ in a similar manner to prove the following claim. We remark that we make use of Claim~\ref{partialsymmetryClaim1} in the proof.

\begin{claim}\label{partialsymmetryClaim2}Assume $m \geq 100(s^2 + s + 1)$.
\begin{itemize}
\item[\textbf{(i)}] There exists a subspace $V^2 \leq V^1$ such that $\dim V^2 \geq \dim V^1 - 2s^2 - 5s$ and $\sum_{i \in [s]} \lambda'_{ij} \alpha_i$ is symmetric for all $j \in [s]$ on $V^2 \times V^2$.
\item[\textbf{(ii)}] For all $i,j \in [s]$, $\lambda_{ij} = \lambda''_{ji}$.
\end{itemize}\end{claim}

\begin{proof}The identity $\phi + \phi \circ (1\,\,3) = 0$ coupled with~\eqref{PhiDecomposition1} gives
\begin{align}0 = &\sum_{i,j \in [s]} (\lambda_{ij} \alpha_i(x,y) \alpha_j(z,w) + \lambda''_{ji} \alpha_i(y,x)\alpha_j(z,w)) + \sum_{i,j \in [s]} \lambda'_{ij} (\alpha_i(x,z) + \alpha_i(z,x))\alpha_j(y,w) \nonumber\\
&\hspace{2cm}+ \sum_{i,j \in [s]} (\lambda''_{ij} \alpha_i(x,w) \alpha_j(y,z) + \lambda_{ji} \alpha_i(x,w) \alpha_j(z,y))\nonumber\\
+&\rho(x,y) (\beta_1(z,w) + \beta_4(z,w)) + \rho(x,w) (\beta_3(y,z) + \beta_6(z,y)) + \rho(y,z) (\beta_4(x,w) + \beta_1(x,w))\nonumber\\
&\hspace{2cm} + \rho(y,w) (\beta_5(x,z) + \beta_5(z,x)) + \rho(z,w) (\beta_6(x,y) + \beta_3(y,x))\label{13decompositionsymmID}\end{align}
for all $x,y,z,w \in U'$.\\

\noindent\textbf{Proof of (i).} Let $j \in [s]$ be given. Similarly to the previous claim, we apply Lemma~\ref{bilCountingLemma} to find $y,w \in V^1$ such that $\rho(y,w) = 0$, $\alpha_i(y,w) = 0$ for $i \not= j$ and $\alpha_j(y,w) = 1$. Identity~\eqref{13decompositionsymmID} for this choice of $y,w$ and Lemma~\ref{rankscharlemma}\textbf{(iii)} show that the rank of the bilinear map $\Big(\sum_{i \in [s]}\lambda'_{ij} (\alpha_i|_{V^1 \times V^1}(x,z) + \alpha_i|_{V^1 \times V^1}(z,x))\Big)$ is at most $2s + 5$. Lemma~\ref{lowranktozero} provides us with a subspace $V^2_j \leq V^1$ with $\dim V_j^2 \geq V^1 - 2s - 5$ on which $\sum_{i \in [s]}\lambda'_{ij} \alpha_i$ is symmetric. To finish the proof, take $V^2 = \cap_{j \in [s]} V^2_j$.\\

\noindent\textbf{Proof of (ii).} Using the part \textbf{(i)} of Claim~\ref{partialsymmetryClaim1} we see that whenever $x,y,z,w \in V^1$ 
\begin{align*}\sum_{i,j \in [s]} (\lambda_{ij} \alpha_i(x,y) \alpha_j(z,w) + \lambda''_{ji} \alpha_i(y,x)\alpha_j(z,w)) = &\sum_{j \in [s]} \Big(\sum_{i \in [s]} \lambda_{ij} \alpha_i(x,y)\Big) \alpha_j(z,w) + \sum_{i,j \in [s]} \lambda''_{ji} \alpha_i(y,x)\alpha_j(z,w)\\
= &\sum_{j \in [s]} \Big(\sum_{i \in [s]} \lambda_{ij} \alpha_i(y,x)\Big) \alpha_j(z,w) + \sum_{i,j \in [s]} \lambda''_{ji} \alpha_i(y,x)\alpha_j(z,w)\\
= &\sum_{i,j \in [s]} (\lambda_{ij} + \lambda''_{ji}) \alpha_i(y,x)\alpha_j(z,w).\end{align*}
From~\eqref{13decompositionsymmID} we obtain
\begin{align}0 = &\sum_{i,j \in [s]} (\lambda_{ij} + \lambda''_{ji}) \alpha_i(y,x)\alpha_j(z,w) + \sum_{i,j \in [s]} \lambda'_{ij} (\alpha_i(x,z) + \alpha_i(z,x))\alpha_j(y,w) \nonumber\\
&\hspace{2cm}+ \sum_{i,j \in [s]} (\lambda''_{ij} \alpha_i(x,w) \alpha_j(y,z) + \lambda_{ji} \alpha_i(x,w) \alpha_j(z,y))\nonumber\\
+&\rho(x,y) (\beta_1(z,w) + \beta_4(z,w)) + \rho(x,w) (\beta_3(y,z) + \beta_6(z,y)) + \rho(y,z) (\beta_4(x,w) + \beta_1(x,w))\nonumber\\
&\hspace{2cm} + \rho(y,w) (\beta_5(x,z) + \beta_5(z,x)) + \rho(z,w) (\beta_6(x,y) + \beta_3(y,x))\label{13decompositionsymmIDsecond}\end{align} 
for all $x,y,z,w \in V^1$. We now proceed as in the proof of Claim~\ref{partialsymmetryClaim1}\textbf{(ii)}. Let $i, j \in [s]$ be given. Since $m \geq 100(s^2 + s + 1)$, we may find elements $x,y,z,w \in V^1$ such that $\alpha_1, \dots, \alpha_r$ and $\rho$ are equal to 0 at all 6 points in the set $\{(y,x), (x,z), (x,w), (y,z), (y,w), (z,w)\}$, with the two exceptions being $\alpha_i(y,x) = \alpha_j(z,w) = 1$. (Note that this time we use the point $(y,x)$ instead of $(x,y)$.) Evaluating~\eqref{13decompositionsymmIDsecond} at $(x,y,z,w)$ gives $\lambda_{ij} = \lambda''_{ji}$.\end{proof}

It remains to use the third symmetry condition $\phi + \phi \circ (1\,\,4) = \rho(x,y) \rho(z,w) + \rho(x,z) \rho(y,w)$.

\begin{claim}\label{partialsymmetryClaim3}Assume $m \geq 200(s^2 + s + 1)$. Then $\lambda_{ij} = \lambda'_{ji}$ for all $i,j \in [s]$.\end{claim}

\begin{proof}The identity $\phi + \phi \circ (1\,\,4) = \rho(x,y) \rho(z,w) + \rho(x,z) \rho(y,w)$ implies that
\begin{align}0=&\sum_{i,j \in [s]} (\lambda_{ij} \alpha_i(x,y) \alpha_j(z,w) + \lambda'_{ji}  \alpha_i(y,x)\alpha_j(w,z)) + \sum_{i,j \in [s]} (\lambda'_{ij} \alpha_i(x,z) \alpha_j(y,w) + \lambda_{ji} \alpha_i(z,x)\alpha_j(w,y))\nonumber\\
&\hspace{2cm} + \sum_{i,j \in [s]} \lambda''_{ij} (\alpha_i(x,w) + \alpha_i(w,x)) \alpha_j(y,z)\nonumber\\
&+\rho(x,y) (\rho(z,w) + \beta_1(z,w)+ \beta_5(w,z)) + \rho(x,z) (\rho(y,w)  + \beta_2(y,w) + \beta_6(w,y))\nonumber\\
&\hspace{2cm}+ \rho(y,z) (\beta_4(x,w) + \beta_4(w,x)) + \rho(y,w) (\beta_5(x,z) + \beta_1(z,x)) + \rho(z,w)(\beta_6(x,y) + \beta_2(y,x))\label{14decompositionsymmID}\end{align}
for all $x,y,z,w \in U'$. Using Claim~\ref{partialsymmetryClaim1}\textbf{(i)} and Claim~\ref{partialsymmetryClaim2}\textbf{(i)} we see that whenever $x,y,z,w \in V^2$ 
\begin{align*}\sum_{i,j \in [s]} (\lambda_{ij} \alpha_i(x,y) \alpha_j(z,w) +& \lambda'_{ji}  \alpha_i(y,x)\alpha_j(w,z))\\
 = & \sum_{j \in [s]} \Big(\sum_{i \in [s]} \lambda_{ij} \alpha_i(x,y)\Big) \alpha_j(z,w) + \sum_{i \in [s]} \alpha_i(y,x)\Big(\sum_{j \in [s]}\lambda'_{ji} \alpha_j(w,z)\Big)\\
= & \sum_{j \in [s]} \Big(\sum_{i \in [s]} \lambda_{ij} \alpha_i(y,x)\Big) \alpha_j(z,w) + \sum_{i \in [s]} \alpha_i(y,x)\Big(\sum_{j \in [s]}\lambda'_{ji} \alpha_j(z,w)\Big)\\
= &\sum_{i,j \in [s]} (\lambda_{ij} + \lambda'_{ji}) \alpha_i(y,x)\alpha_j(z,w).\end{align*}

Using this identity, the expression in~\eqref{14decompositionsymmID} becomes
\begin{align*}0=&\sum_{i,j \in [s]} (\lambda_{ij} + \lambda'_{ji}) \alpha_i(y,x)\alpha_j(z,w) + \sum_{i,j \in [s]} (\lambda'_{ij} + \lambda_{ji}) \alpha_i(z,x)\alpha_j(y,w) + \sum_{i,j \in [s]} \lambda''_{ij} (\alpha_i(x,w) + \alpha_i(w,x)) \alpha_j(y,z)\nonumber\\
&+\rho(x,y) (\rho(z,w) + \beta_1(z,w)+ \beta_5(w,z)) + \rho(x,z) (\rho(y,w)  + \beta_2(y,w) + \beta_6(w,y))\nonumber\\
&\hspace{2cm}+ \rho(y,z) (\beta_4(x,w) + \beta_4(w,x)) + \rho(y,w) (\beta_5(x,z) + \beta_1(z,x)) + \rho(z,w)(\beta_6(x,y) + \beta_2(y,x))\label{14decompositionsymmIDalt}\end{align*}
for all $x,y,z,w \in V^2$. Let $i,j \in [s]$ be fixed now. We complete the argument as in the proofs of the previous two claims. Since $m$ is sufficiently large, using Lemma~\ref{bilCountingLemma} we may find elements $x,y,z,w \in V^2$ such that $\alpha_1, \dots, \alpha_s$ and $\rho$ are equal to 0 at all 6 points in the set $\{(y,x), (z,x), (w,x), (y,z), (y,w), (z,w)\}$, with the two exceptions being $\alpha_i(y,x) = \alpha_j(z,w) = 1$. It follows that $\lambda_{ij} = \lambda'_{ji}$, as required.\end{proof}

It follows from Claims~\ref{partialsymmetryClaim1}\textbf{(ii)},~\ref{partialsymmetryClaim2}\textbf{(ii)} and \ref{partialsymmetryClaim3} that 
\[\lambda_{ij} = \lambda'_{ji} = \lambda''_{ij} = \lambda_{ji}\]
so $\lambda_{ij} = \lambda'_{ij} = \lambda''_{ij}$ and $\lambda$ is symmetric. Thus, returning to equality~\eqref{PhiDecomposition1} we see that whenever $x,y,z,w \in U'$ then
\begin{align*}\phi(x,y,z,w) = \sigma(x,y,z,w) + \psi(x,y,z,w) &+ \rho(x,y) \beta_1(z,w) + \rho(x,z) \beta_2(y,w) + \rho(x,w) \beta_3(y,z)
\\&  + \rho(y,z) \beta_4(x,w) + \rho(y,w) \beta_5(x,z) + \rho(z,w) \beta_6(x,y)\end{align*}
where $\psi \colon U' \times  U' \times U' \times U' \to \mathbb{F}_2$ is the multilinear form defined as 
\[\psi(x,y,z,w) = \sum_{i,j \in [s]} \lambda_{ij} \alpha_i(x,y) \alpha_j(z,w) + \sum_{i,j \in [s]} \lambda_{ij} \alpha_i(x,z) \alpha_j(y,w) + \sum_{i,j \in [s]} \lambda_{ij} \alpha_i(x,w) \alpha_j(y,z).\]

\vspace{\baselineskip}

\begin{claim}Multilinear form $\psi|_{V^1 \times V^1 \times V^1 \times V^1}$ is symmetric.\end{claim}

\begin{proof} Before we proceed with the proof, we observe a few useful identities that hold for any $a,b,c,d \in V^1$. Using Claim~\ref{partialsymmetryClaim1}\textbf{(ii)} and the fact that $\lambda$ is symmetric we have
\begin{align}\sum_{i,j \in [s]} \lambda_{ij} \alpha_i(a,b)\alpha_j(c,d) = \sum_{j \in [s]} \Big(\sum_{i \in [s]} \lambda_{ij} \alpha_i(a,b)\Big)\alpha_j(c,d) = &\sum_{j \in [s]} \Big(\sum_{i \in [s]} \lambda_{ij} \alpha_i(b,a)\Big) \alpha_j(c,d)\nonumber\\
 = &\sum_{i,j \in [s]} \lambda_{ij} \alpha_i(b,a)\alpha_j(c,d)\label{abcdSym1ID}\end{align}
and
\begin{align}\sum_{i,j \in [s]} \lambda_{ij} \alpha_i(a,b)\alpha_j(c,d) = &\sum_{j,i \in [s]} \lambda_{ji} \alpha_j(a,b)\alpha_i(c,d) = \sum_{i,j \in [s]} \lambda_{ij} \alpha_i(c,d)\alpha_j(a,b).\label{abcdSym2ID}\end{align}

Additionally, from these two identities we deduce
\begin{equation}\sum_{i,j \in [s]} \lambda_{ij} \alpha_i(a,b)\alpha_j(c,d) = \sum_{i,j \in [s]} \lambda_{ij} \alpha_i(c,d)\alpha_j(a,b) = \sum_{i,j \in [s]} \lambda_{ij} \alpha_i(d,c)\alpha_j(a,b) = \sum_{i,j \in [s]} \lambda_{ij} \alpha_i(a,b)\alpha_j(d,c).\label{abcdSym3ID}\end{equation}

Returning to the claim, it suffices to prove $\psi = \psi \circ (1\,\,2) = \psi \circ (1\,\,3) = \psi \circ (1\,\,4)$ on $V^1 \times V^1 \times V^1 \times V^1$. Fix $x,y,z,w \in V^1$. Using the idenities~\eqref{abcdSym1ID},~\eqref{abcdSym2ID} and~\eqref{abcdSym3ID} above we obtain
\begin{align*}\psi(y,x,z,w) = &\sum_{i,j \in [s]} \lambda_{ij} \alpha_i(y,x) \alpha_j(z,w) + \sum_{i,j \in [s]} \lambda_{ij} \alpha_i(y,z) \alpha_j(x,w) + \sum_{i,j \in [s]} \lambda_{ij} \alpha_i(y,w) \alpha_j(x,z)\\
= &\sum_{i,j \in [s]} \lambda_{ij} \alpha_i(x,y) \alpha_j(z,w) + \sum_{i,j \in [s]} \lambda_{ij} \alpha_i(x,w)\alpha_j(y,z) + \sum_{i,j \in [s]} \lambda_{ij} \alpha_i(x,z)\alpha_j(y,w)\\
= &\sum_{i,j \in [s]} \lambda_{ij} \alpha_i(x,y) \alpha_j(z,w)+ \sum_{i,j \in [s]} \lambda_{ij} \alpha_i(x,z)\alpha_j(y,w) + \sum_{i,j \in [s]} \lambda_{ij} \alpha_i(x,w)\alpha_j(y,z)\\
= & \psi(x, y,z,w).\end{align*}

\noindent Next, we have
\begin{align*}\psi(z,y,x,w) =&\sum_{i,j \in [s]} \lambda_{ij} \alpha_i(z,y) \alpha_j(x,w) + \sum_{i,j \in [s]} \lambda_{ij} \alpha_i(z,x) \alpha_j(y,w) + \sum_{i,j \in [s]} \lambda_{ij} \alpha_i(z,w) \alpha_j(y,x)\\
=& \sum_{i,j \in [s]} \lambda_{ij} \alpha_i(y,z) \alpha_j(x,w) + \sum_{i,j \in [s]} \lambda_{ij} \alpha_i(x,z) \alpha_j(y,w) + \sum_{i,j \in [s]} \lambda_{ij} \alpha_i(z,w) \alpha_j(x,y)\\
=&\sum_{i,j \in [s]} \lambda_{ij} \alpha_i(x,w)\alpha_j(y,z) + \sum_{i,j \in [s]} \lambda_{ij} \alpha_i(x,z) \alpha_j(y,w) + \sum_{i,j \in [s]} \lambda_{ij} \alpha_i(x,y)\alpha_j(z,w)\\
=&\sum_{i,j \in [s]} \lambda_{ij} \alpha_i(x,y)\alpha_j(z,w) + \sum_{i,j \in [s]} \lambda_{ij} \alpha_i(x,z) \alpha_j(y,w) + \sum_{i,j \in [s]} \lambda_{ij} \alpha_i(x,w)\alpha_j(y,z)\\
= & \psi(x, y,z,w).\end{align*}

\noindent Finally, we have 
\begin{align*}\psi(w,y,z,x) =&\sum_{i,j \in [s]} \lambda_{ij} \alpha_i(w,y) \alpha_j(z,x) + \sum_{i,j \in [s]} \lambda_{ij} \alpha_i(w,z) \alpha_j(y,x) + \sum_{i,j \in [s]} \lambda_{ij} \alpha_i(w,x) \alpha_j(y,z)\\
= &\sum_{i,j \in [s]} \lambda_{ij} \alpha_i(y,w) \alpha_j(x,z) + \sum_{i,j \in [s]} \lambda_{ij} \alpha_i(z,w) \alpha_j(x,y) + \sum_{i,j \in [s]} \lambda_{ij} \alpha_i(x,w) \alpha_j(y,z)\\
= &\sum_{i,j \in [s]} \lambda_{ij}  \alpha_i(x,z)\alpha_j(y,w) + \sum_{i,j \in [s]} \lambda_{ij} \alpha_i(x,y)\alpha_j(z,w)  + \sum_{i,j \in [s]} \lambda_{ij} \alpha_i(x,w) \alpha_j(y,z)\\
=  &\sum_{i,j \in [s]} \lambda_{ij} \alpha_i(x,y)\alpha_j(z,w)  + \sum_{i,j \in [s]} \lambda_{ij}  \alpha_i(x,z)\alpha_j(y,w) + \sum_{i,j \in [s]} \lambda_{ij} \alpha_i(x,w) \alpha_j(y,z)\\
= & \psi(x, y,z,w).\qedhere\end{align*}
\end{proof}

Thus, we may consume $\psi$ into $\sigma$ and without loss of generality $\phi$ takes the shape
\begin{align}\phi(x,y,z,w) = \sigma(x,y,z,w) + &\rho(x,y) \beta_1(z,w) + \rho(x,z) \beta_2(y,w) + \rho(x,w) \beta_3(y,z)\nonumber\\ 
+ &\rho(y,z) \beta_4(x,w) + \rho(y,w) \beta_5(x,z) + \rho(z,w) \beta_6(x,y)\label{PhiDecomposition2}\end{align}
for $x,y,z,w \in V^1$.\\

We use the symmetry properties of $\phi$ for the second time in order to elucidate the structure of forms $\beta_1, \dots, \beta_6$. Equality $\phi = \phi \circ (1\,\,2)$ implies that
\begin{align*}0 =& \phi(x,y,z,w) + \phi(y,x,z,w) \\
= &\rho(x,z) (\beta_2(y,w) + \beta_4(y,w))+ \rho(x,w) (\beta_3(y,z) + \beta_5(y,z)) + \rho(y,z) (\beta_4(x,w) + \beta_2(x,w))\\
 &\hspace{2cm} + \rho(y,w) (\beta_5(x,z) + \beta_3(x,z)) + \rho(z,w) (\beta_6(x,y) + \beta_6(y,x))\end{align*}
for $x,y,z,w \in V^1$. If we fix $(z,w) \in V^1 \times V^1$ such that $\rho(z,w) = 1$ (such a point exists by Lemma~\ref{bilCountingLemma} if $m \geq 100(s^2 + s + 1)$), we see that the rank of the map $(x,y) \mapsto \beta_6(x,y) + \beta_6(y,x)$ on $V^1 \times V^1$ is at most 4. Similarly, if we fix $(x,z) \in V^1 \times V^1$ such that $\rho(x,z) = 1$ (respectively $(x,w) \in V^1 \times V^1$ such that $\rho(x,w) = 1$), we obtain that the rank of $\beta_2 + \beta_4 + \mu \rho$ for scalar $\mu = \beta_3(x,z) + \beta_5(x,z)$ (respectively $\beta_3 + \beta_5 + \mu' \rho$ for scalar $\mu' = \beta_2(x,w) + \beta_4(x,w)$) is at most 3 on $V^1 \times V^1$. Using Lemma~\ref{lowranktozero} we find a suitable subspace $W^1 \leq V^1$ of $\dim W^1 \geq \dim V^1 - 10$ such that $\beta_6|_{W^1 \times W^1}$ is symmetric and
\begin{equation}\beta_2|_{W^1 \times W^1} + \beta_4|_{W^1 \times W^1}, \beta_3|_{W^1 \times W^1} + \beta_5|_{W^1 \times W^1} \in \langle \rho|_{W^1 \times W^1} \rangle.\label{betaRelationship1}\end{equation}

Proceeding further, equality $\phi = \phi \circ (1\,\,3)$ and the fact that $\beta_6|_{W^1 \times W^1}$ is symmetric imply that for $x,y,z,w \in W^1$ we have
\begin{align*}0 =& \phi(x,y,z,w) + \phi(z,y,x,w) \\
= & \rho(x,y) (\beta_1(z,w) + \beta_4(z,w)) + \rho(x,w) (\beta_3(y,z) + \beta_6(z,y)) + \rho(y,z) (\beta_4(x,w) + \beta_1(x,w))\\
&\hspace{2cm}+ \rho(y,w) (\beta_5(x,z) + \beta_5(z,x)) + \rho(z,w) (\beta_6(y,x) + \beta_3(x,y))\\
= & \rho(x,y) (\beta_1(z,w) + \beta_4(z,w)) + \rho(x,w) (\beta_3(y,z) + \beta_6(y,z)) + \rho(y,z) (\beta_4(x,w) + \beta_1(x,w))\\
&\hspace{2cm}+ \rho(y,w) (\beta_5(x,z) + \beta_5(z,x)) + \rho(z,w) (\beta_6(x,y) + \beta_3(x,y)).\end{align*}
Again, fixing suitable pairs of $x,y,z,w \in W^1$ with $\rho = 1$ (using Lemma~\ref{bilCountingLemma}, assuming $m \geq 200(s^2 + s + 1)$), we see that the ranks of $(x,z) \mapsto \beta_5(x,z) + \beta_5(z,x)$, $\beta_1 + \beta_4 + \mu\rho$ and $\beta_3 + \beta_6 + \mu'\rho$ for suitable scalars $\mu, \mu' \in \mathbb{F}_2$ are at most 4,3 and 3 respectively. After passing to a suitable subspace $W^2 \leq W^1$ of $\dim W^2 \geq \dim W^1 - 10$ using Lemma~\ref{lowranktozero}, we obtain that $\beta_5|_{W^2 \times W^2}$ is symmetric and 
\begin{equation}\beta_1|_{W^2 \times W^2} + \beta_4|_{W^2 \times W^2}, \beta_3|_{W^2 \times W^2} + \beta_6|_{W^2 \times W^2} \in \langle \rho|_{W^2 \times W^2} \rangle.\label{betaRelationship2}\end{equation}

Finally, equality $\phi(x,y,z,w) = \phi(w,y,z,x) + \rho(x,y) \rho(z,w) + \rho(x,z) \rho(y,w)$ and symmetry of maps $\beta_6|_{W^2 \times W^2}$ and $\beta_5|_{W^2 \times W^2}$ give
\begin{align*}&\rho(x,y) \rho(z,w) + \rho(x,z) \rho(y,w) = \phi(x,y,z,w) + \phi(w,y,z,x)\\
&\hspace{2cm}=\rho(x,y) (\beta_1(z,w) + \beta_5(w,z)) + \rho(x,z)(\beta_2(y,w)+\beta_6(w,y)) + \rho(y,z)(\beta_4(x,w) + \beta_4(w,x))\\
&\hspace{4cm}+ \rho(y,w) (\beta_5(x,z) + \beta_1(z, x)) + \rho(z,w) (\beta_6(x,y) + \beta_2(y,x))\\
&\hspace{2cm}=\rho(x,y) (\beta_1(z,w) + \beta_5(z,w)) + \rho(x,z)(\beta_2(y,w)+\beta_6(y,w)) + \rho(y,z)(\beta_4(x,w) + \beta_4(w,x))\\
&\hspace{4cm}+ \rho(y,w) (\beta_5(z,x) + \beta_1(z, x)) + \rho(z,w) (\beta_6(y,x) + \beta_2(y,x))\end{align*}
for all $x,y,z,w \in W^2$. Applying the same argument for the third time, after fixing suitable pairs of $x,y,z,w \in W^1$ with $\rho = 1$ (using Lemma~\ref{bilCountingLemma}, this time assuming $m \geq 300(s^2 + s + 1)$) we see that the ranks of $(x,w) \mapsto \beta_4(x,w) + \beta_4(w,x)$, $\beta_1 + \beta_5 + \mu\rho$ and $\beta_2 + \beta_6 + \mu'\rho$ for suitable scalars $\mu, \mu' \in \mathbb{F}_2$ are at most 4,3 and 3 respectively. We thus find a subspace $W^3 \leq W^2$ of $\dim W^3 \geq \dim W^2 - 10$ using Lemma~\ref{lowranktozero}, such that $\beta_4|_{W^3 \times W^3}$ is symmetric and
\begin{equation}\beta_1|_{W^3 \times W^3} + \beta_5|_{W^3 \times W^3}, \beta_2|_{W^3 \times W^3} + \beta_6|_{W^3 \times W^3} \in \langle \rho|_{W^3 \times W^3} \rangle.\label{betaRelationship3}\end{equation}
\vspace{\baselineskip}

Using facts~\eqref{betaRelationship1},~\eqref{betaRelationship2} and~\eqref{betaRelationship3}, writing $\beta = \beta_6$ and restricting all maps to the product $W^3 \times W^3$, we obtain $\beta_2, \beta_3 \in \beta + \langle \rho \rangle$, from which we further get $\beta_4 \in \beta_2 + \langle \rho \rangle = \beta + \langle \rho \rangle$ and $\beta_5 \in \beta_3 + \langle \rho \rangle = \beta + \langle \rho \rangle$, which finally implies $\beta_1 \in \beta_4 + \langle \rho \rangle = \beta + \langle \rho \rangle$. Hence, all $\beta_i$ belong to $\beta + \langle \rho \rangle$. In conclusion, we obtain scalars $\lambda_1, \lambda_2, \lambda_3 \in \mathbb{F}_2$ such that
\begin{align}\phi(x,y,z,w) = \sigma(x,y,z,w) + &\rho(x,y) \beta(z,w) + \rho(x,z) \beta(y,w) + \rho(x,w) \beta(y,z)\nonumber\\ 
+ &\rho(y,z) \beta(x,w) + \rho(y,w) \beta(x,z) + \rho(z,w) \beta(x,y)\nonumber\\
+ &\lambda_1\rho(x,y)\rho(z,w) + \lambda_2\rho(x,z)\rho(y,w) + \lambda_3\rho(x,w)\rho(y,z)\nonumber\end{align}
holds for all $x,y,z,w \in W^3$.\\

Since $\rho$ and $\beta$ are symmetric on $W^3 \times W^3$, the multilinear form 
\[\rho(x,y) \beta(z,w) + \rho(x,z) \beta(y,w) + \rho(x,w) \beta(y,z) + \rho(y,z) \beta(x,w) + \rho(y,w) \beta(x,z) + \rho(z,w) \beta(x,y)\]
is readily seen to be symmetric, so it can be consumed into $\sigma$, and on $W^3 \times W^3 \times W^3 \times W^3$ we may assume that
\begin{align}\phi(x,y,z,w) = \sigma(x,y,z,w) + \lambda_1\rho(x,y)\rho(z,w) + \lambda_2\rho(x,z)\rho(y,w) + \lambda_3\rho(x,w)\rho(y,z).\label{PhiDecomposition3}\end{align}
\vspace{\baselineskip}

Let us use the symmetry conditions for $\phi$ for the final time in order to understand the relationship between scalars $\lambda_1, \lambda_2$ and $\lambda_3$. Note that the codimension of the subspace $W^3$ inside $U'$ is at most $2s^2 + 5s + 30$, so by Lemma~\ref{rankdecsub} we have $\on{rank} \rho|_{W^3 \times W^3} \geq 2$ (assuming $m \geq 100(s^2 + s+ 1)$). The equality $\phi = \phi \circ (1\,\,2)$ gives
\[(\lambda_2 + \lambda_3)(\rho(x,z)\rho(y,w) + \rho(x,w)\rho(y,z)) = 0\]
for all $x,y,z,w \in W^3$. Since $\on{rank} \rho|_{W^3 \times W^3} \geq 2$, the map $\rho|_{W^3 \times W^3}$ is non-zero, so we can find $(y,w) \in W^3 \times W^3$ such that $\rho(y,w) = 1$. The equality above then shows that $(\lambda_2 + \lambda_3)\rho|_{W^3 \times W^3}$ has rank at most 1. However, using the fact that the rank of $\rho|_{W^3 \times W^3}$ is at least 2 again, we deduce $\lambda_2 = \lambda_3$.\\
\indent Similarly, the condition $\phi = \phi \circ (1\,\,3)$ implies 
\[(\lambda_1 + \lambda_3)(\rho(x,y)\rho(z,w) + \rho(x,w)\rho(y,z)) = 0\]
for all $x,y,z,w \in W^3$. Using the same rank bound $\on{rank} \rho|_{W^3 \times W^3} \geq 2$ allows us to deduce $\lambda_1 = \lambda_3$.\\
\indent Putting these equalities together we see that $\lambda_1 = \lambda_2 = \lambda_3$ and equality~\eqref{PhiDecomposition3} then becomes
\[\phi(x,y,z,w) = \sigma(x,y,z,w) + \lambda_1\Big(\rho(x,y)\rho(z,w) + \rho(x,z)\rho(y,w) + \rho(x,w)\rho(y,z)\Big)\]
so $\phi = \phi \circ (1\,\,4)$ on $W^3 \times W^3$. The third symmetry condition tells us that
\[\rho(x,y) \rho(z,w) + \rho(x,z) \rho(y,w) = 0\]
for all $x,y,z,w \in W^3$, which is a contradiction as the rank of $\rho|_{W^3 \times W^3}$ is at least 2. To complete the proof of the theorem it remains to choose $m = 20000(r^2 + r + 1)$ (recall that $s \leq 6r$) so that all conditions on $m$ are satisfied.\end{proof}

\end{document}